\documentclass{amsart}
\usepackage{xcolor}
\usepackage[colorlinks]{hyperref}
\hypersetup{colorlinks=true,linkcolor=red, anchorcolor=green,
	citecolor=cyan, urlcolor=red, filecolor=magenta, pdftoolbar=true}
 \newtheorem{thm}{Theorem}[section]
 
 \newtheorem{lem}[thm]{Lemma}
 \newtheorem{prop}[thm]{Proposition}
 \theoremstyle{definition}
 
 \theoremstyle{remark}
 \newtheorem{rem}[thm]{Remark}
 
 \numberwithin{equation}{section}

\begin{document}

%
%
%
%
%
%
%
%
%

\title[The quintic complex moment problem]
 {The quintic complex moment problem}

\author{H. El-Azhar}
\address{%
Mohammed V University in Rabat.\\
Rabat\\
Morocco} \email{elazharhamza@gmail.com}
\author{A. Harrat}
\address{%
Mohammed V University in Rabat.\\
Rabat\\
Morocco} \email{ayoub1993harrat@gmail.com }
\author{K. Idrissi}
\address{%
Mohammed V University in Rabat.\\
Rabat\\
Morocco} \email{kaissar.idrissi@gmail.com}
\author[E. H. Zerouali]{E. H. Zerouali}

\address{%
Mohammed V University in Rabat.\\
Rabat\\
Morocco}

\email{zerouali@fsr.ac.ma}



\date{}

\begin{abstract}
Let $\gamma^{(m)} \equiv \{ \gamma_{ij} \}_{0 \leq i +j \leq m}$ be
a given complex-valued sequence. The truncated complex moment
problem (TCMP in short) involves determining necessary and
sufficient conditions for the existence of a positive Borel measure
$\mu$ on $\mathbb{C}$ (called a representing measure for
$\gamma^{(m)}$) such that $\gamma_{ij} = \int \overline{z}^i z^j
d\mu$ for $0 \leq i +j \leq m$. The TCMP has been completely solved
only when $m= 1, 2, 3, 4$.

 We provide in this paper a concrete solution to the quintic TCMP (that is, when $m = 5$). We also study  the cardinality of the minimal representing
measure. Based on the bivariate recurrences sequences's properties
with some Curto-Fialkow's results, our method intended to be useful
for all odd-degree moment problems.
\end{abstract}

\maketitle
\section{Introduction}
Given a doubly indexed finite sequence of complex numbers
 $$\gamma \equiv \gamma^{(m)} = \{\gamma_{ij}\}_{0 \leq i +j \leq m} = \{\gamma_{00}, \gamma_{01}, \gamma_{10}, \ldots, \gamma_{0m}, \ldots,  \gamma_{m0} \}$$
with $\gamma_{00}>0$ and $\overline{\gamma}_{ij}=\gamma_{ji} $ for
$0 \leq i +j \leq m$. The truncated complex moment problem (in
short, TCMP) associated with  $\gamma$  entails finding a positive
Borel measure $\mu$ supported in the complex plane $\mathbb{C}$ such
that
 \begin{equation}\label{i.0}
\gamma_{ij}=\int \overline{z}^iz^jd\mu \phantom{du texte} (0\leq i+
j\leq m);
\end{equation}
A sequence  $\{\gamma_{ij}\}_{0\leq i+ j\leq m}$   satisfying
(\ref{i.0})   will be  called a truncated moment sequence and the
solution $\mu$ is said to be a representing measure associated to
the sequence  $\{\gamma_{ij}\}_{0\leq i+ j\leq m}$.

In \cite{Stochel2001solving} J. Stochel  has shown that solving TCMP
solves the  widely studied  Full Moment Problem (see, for example,
\cite{Aheizer1965classical, Bayer2006proof,
Blekherman2012nonnegative, Demanze2000moments, Laurent2009sums,
Putinar2006multivariate, Shohat1943problem,
Vasilescu2012dimensional}). More precisely, a full moment sequence
$\{ \gamma_{ij} \}_{i, j \in \mathbb{Z}_+}$ admits a representing
measure if and only if each of its  truncation $\gamma^{(m)}$ admits
a representing measure.

The truncated complex moment problem serves as a prototype for
several other moment problems  to which it is closely related. Its
application can be found in  subnormal operator theory
\cite{sarason1987moment, lambert1984characterization,
tchakaloff1957formules}, polynomial hyponormality
\cite{curto1993nearly} and joint hyponormality
\cite{curto1993recursively, curto1994recursively}. It is also
related to the optimization theory \cite{Lasserre2001global,
Lasserre2002polynomials, Lasserre2002bounds, lasserre2010moments,
Laurent2009sums} and arise in pure and applied mathematics and in
the sciences in general.

For the  even case  $m = 2n$, Curto and Fialkow developed in a
series of papers  an approach for TCMP based on positivity and flat
extensions of the moment matrix, see Section 2. This allowed them to
find solutions for various particular cases of  truncated moment
problems (see, for instance, \cite{CURTO1996SOLUTION,
CURTO1998FLATCONJUGATE, CURTO1998FLAT, curto2005truncated,
curto2008extremal,idrissi2016multivariable,idrissi2016complex}).
However, only the cases $m=2$ and $m = 4$ are completely solved (cf.
\cite{CURTO1996SOLUTION, curto2002solution, fialkow2010positivity,
curto2016concrete}).

In the odd case $m =2n +1$, a general solution to some partial cases
of the TCMP can be found in \cite{kimsey2011matrix} and
\cite{kimsey2014cubic} as well as a solution to the truncated matrix
moment problem; a solution to the cubic complex moment problem (when
m = 3) was given in \cite{kimsey2014cubic}, see also \cite{curto2017new}. The solution is based on
commutativity conditions of matrices determined by $\{ \gamma_{ij}
\}_{0 \leq i+j \leq 2n +1}$.

Therefore, only the cases $m = 1, 2, 3$ and $4$ (the quadratic, the
cubic and the quartic moment problem) have been completely achieved.
All the other cases (quintic, sixtic, ...) are open and interest
several authors; as indicated in many recent papers (see, for
instance, \cite{curto2015non, curto2017division, curto2017new,
Yoo2010extremal, Yoo2017sextic}).

In this paper, we provide a concrete solution to the, almost all, quintic moment
problem (i.e. $m =5$) when one desires a minimal representing
measure. To this aim, we investigate the structure of recursive
complex-valued bi-indexed sequences and we combine the obtained
observations with some results due to R.Curto and L. Fialkow, to
provide a new technique for solving the odd-degree TCMP. We notice
that our techniques furnish a short solution to the cubic moment
problem (we omit the proof because the cubic moment problem is
already solved, see \cite{curto2017new, kimsey2014cubic}) and expected to be
useful for higher odd-degree truncated moment problems.

Let $\gamma^{(5)} = \{\gamma_{ij}\}_{0 \leq i +j \leq 5}$ be a given
complex valued bi-sequence. We associate with $\gamma^{(5)}$ the next
two matrices that will play a crucial role in our approach.
\small{\begin{equation}\label{i.1} M(2):=\left(
       \begin{array}{cccccc}
         \gamma_{00} & \gamma_{01} & \gamma_{10}  & \gamma_{02} & \gamma_{11} & \gamma_{20}\\
         \gamma_{10} & \gamma_{11} & \gamma_{20}  & \gamma_{12} & \gamma_{21} & \gamma_{30}\\
         \gamma_{01} & \gamma_{02} & \gamma_{11}  & \gamma_{03} & \gamma_{12} & \gamma_{21}\\
         \gamma_{20} & \gamma_{21} & \gamma_{30}  & \gamma_{22} & \gamma_{31} & \gamma_{40}\\
         \gamma_{11} & \gamma_{12} & \gamma_{21}  & \gamma_{13} & \gamma_{22} & \gamma_{31}\\
         \gamma_{02} & \gamma_{03} & \gamma_{12}  & \gamma_{04} & \gamma_{13} & \gamma_{22}\\
       \end{array}
     \right)
\hfill B := \left(
       \begin{array}{cccc}
         \gamma_{03} & \gamma_{12} & \gamma_{21} & \gamma_{30}\\
         \gamma_{13} & \gamma_{22} & \gamma_{31} & \gamma_{40}\\
         \gamma_{04} & \gamma_{13} & \gamma_{22} & \gamma_{31}\\
         \gamma_{23} & \gamma_{32} & \gamma_{41} & \gamma_{50}\\
         \gamma_{14} & \gamma_{23} & \gamma_{32} & \gamma_{41}\\
         \gamma_{05} & \gamma_{14} & \gamma_{23} & \gamma_{32}\\
       \end{array}
     \right).
\end{equation}
}

 Let us recall that thanks to Douglas factorization theorem, we have
$Rang~~B \subseteq Rang~~M(2)$ if, and only if, there exists a
matrix $W$ such that $B = M(2) W$. We will show, in Section 2, that
the Hermitian matrix  $W^* M(2) W$ is symmetric with respect to the
second diagonal, then one can set
\begin{equation}\label{i.3}
 W^* M(2) W = \left( \begin{matrix}
 a            & b            & c & d\\
 \overline{b} & e            & f & c\\
 \overline{c} & \overline{f} & e & b\\
 \overline{d} & \overline{c} & \overline{b}     & a
 \end{matrix} \right)
 \end{equation}

As we will see  in the sequel, the entries $a, b, e$ and $f$ in the
matrix $W^* M(2) W$ encodes the  complete information on  the
cardinal of the support of the minimal representing measure.
\begin{thm}\label{main2}
Let $\gamma^{(5)} \equiv \{\gamma_{ij}\}_{i+j\leq 5}$ be a given finite  sequence, such that\\
\centerline{ $M(2) \geq 0$,  $Rang~~B \subseteq Rang~~M(2)$ and $a \neq e$ or $b = f$.}
Then the quintic moment problem, associated with  $\gamma^{(5)}$,
admits a solution $\mu$. Moreover, The smallest cardinality of
$supp~~\mu$ is
  \begin{itemize}
    \item  $card~~supp~~\mu = r \iff a = e$ and $b = f$,
    \item  $card~~supp~~\mu = r +1 \iff  \quad a\neq e \textrm{ and } \frac{a-e}{2} < \mid b -f
    \mid$,
    \item  $card~~supp~~\mu = r +2 \iff \quad a>e \textrm{ and } \frac{a-e}{2} \geq \mid b -f
    \mid$;
  \end{itemize}
      where $r :=~~card~~M(2)$ and the numbers $a, b,e $ and $f$ are given by \eqref{i.3}.
    \end{thm}

Since (as we will show in  Section 2) $M(2) \geq 0$ and $Rang~~B \subseteq Rang~~M(2)$ are two necessary conditions for the quintic TCMP, associated with $\gamma^{(5)}$, then Theorem \ref{main2} provides a concrete solution to the quintic complex moment problem, exept for the case $a \neq e$ or $b = f$. The difficulty that we encountered in solving the remaining case ( $a \neq e$ or $b = f$) is technical, not a failure in the method, see Section 5.

This paper is organized as follows. In Section 2, we will give
useful tools and results usually used in the treatment of the
truncated complex moment problems. We will investigate in Section 3
the complex-valued recursive bi-sequences and we will exhibit
important results  for quintic TCMP in Section 4. Finally, in
Section 5, we solve the quintic complex moment problem together with
the minimal support problem.
\section{ Preliminaries }
First, we recall a fundamental necessary condition. To this end, let
us assume that $\gamma^{(2n)} \equiv \{\gamma_{ij}\}_{i+ j \leq 2n}$
is a given moment sequence and let $\mu$ be the associated
representing measure, then, for every $p \equiv \sum\limits_{h, k}
a_{hk} \overline{z}^h z^k  \in \mathbb{C}[\overline{z}, z]$,
$$0 \leq \int \mid p\mid^2 d\mu = \sum\limits_{h, k, h', k'} a_{hk} \overline{a_{h'k'}} \int \overline{z}^{h +k'} z^{k +h'} = \sum\limits_{h, k, h', k'} a_{hk} \overline{a_{h'k'}} \gamma_{h +k', k +h'},$$
or, equivalently, the moment matrix $M(n)\equiv
M(n)(\gamma^{(2n)})$, defined below, is semi-definite positive.
\begin{equation}\label{moment matrix}
M(n) :=\left( \begin{matrix}
 M[0,0]&M[0,1]&\ldots&M[0,n]\\
 M[1,0]&M[1,1]&\ldots&M[1,n]\\
 \vdots&\vdots&\ddots&\vdots\\
 M[n,0]&M[n,1]&\ldots&M[n,n]
 \end{matrix} \right),
\end{equation}
where
\begin{equation*}
M[i,j]=\left( \begin{matrix}
 \gamma_{i,j}&\gamma_{i+1,j-1}&\ldots&\gamma_{i+j,0}\\
 \gamma_{i-1,j+1}&\gamma_{i,j}&\ldots&\gamma_{i+j-1,1}\\
 \vdots&\vdots&\ddots&\vdots\\
 \gamma_{0,i+j}&\gamma_{1,i+j-1}&\ldots&\gamma_{j,i}
 \end{matrix} \right).
\end{equation*}

Considering the lexicographic order,
\begin{equation}\label{lexicographic order}
1, Z, \overline{Z}, Z^2, Z\overline{Z}, \overline{Z}^2, \dots, Z^n
,Z^{n-1}\overline{Z}, \ldots ,Z\overline{Z}^{n-1}, \overline{Z}^n,
\end{equation}
to denote rows and columns of the moment matrix $M(n)$.
 For example, The $M(3)$ matrix is
\small{
\begin{equation}\label{M(4)}
\bordermatrix{
         &1          & & Z    & \overline{Z} & & Z^2  & Z\overline{Z} & \overline{Z}^2 & & Z^3 & Z^2 \overline{Z} & Z \overline{Z}^2 & \overline{Z}^3
\cr
         1 &\gamma_{00} & \mid & \gamma_{01} & \gamma_{10}  & \mid & \gamma_{02} & \gamma_{11} & \gamma_{20} & \mid & \gamma_{03} & \gamma_{12} & \gamma_{21} & \gamma_{30}
\cr
         &--        &-& --   & -- & - & --  & -- & -- & - & -- & -- &  -- & --
\cr
         Z &\gamma_{10} & \mid & \gamma_{11} & \gamma_{20}  & \mid & \gamma_{12} & \gamma_{21} & \gamma_{30} & \mid & \gamma_{13} & \gamma_{22} & \gamma_{31} & \gamma_{40}
         \cr
         \overline{Z} &\gamma_{01} & \mid & \gamma_{02} & \gamma_{11}  & \mid & \gamma_{03} & \gamma_{12} & \gamma_{21} & \mid & \gamma_{04} & \gamma_{13} & \gamma_{22} & \gamma_{31}
\cr
         &--        &-& --   & -- & - & --  & -- & -- & - & -- & -- &  -- & --
\cr
         Z^2  &\gamma_{20} & \mid & \gamma_{21} & \gamma_{30}  & \mid & \gamma_{22} & \gamma_{31} & \gamma_{40} & \mid & \gamma_{23} & \gamma_{32} & \gamma_{41} & \gamma_{50}
         \cr
         Z\overline{Z} &\gamma_{11} & \mid & \gamma_{12} & \gamma_{21}  & \mid & \gamma_{13} & \gamma_{22} & \gamma_{31} & \mid & \gamma_{14} & \gamma_{23} & \gamma_{32} & \gamma_{41}\cr
         \overline{Z}^2 &\gamma_{02} & \mid & \gamma_{03} & \gamma_{12}  & \mid & \gamma_{04} & \gamma_{13} & \gamma_{22} & \mid & \gamma_{05} & \gamma_{14} & \gamma_{23} & \gamma_{32}
\cr
         &--        &-& --   & -- & - & --  & -- & -- & - & -- & -- &  -- & --
\cr
         Z^3 &\gamma_{30} & \mid & \gamma_{31} & \gamma_{40}  & \mid & \gamma_{32} & \gamma_{41} & \gamma_{50} & \mid & \gamma_{33} & \gamma_{42} & \gamma_{51} & \gamma_{60}
         \cr
         Z^2 \overline{Z} &\gamma_{21} & \mid & \gamma_{22} & \gamma_{31}  & \mid & \gamma_{23} & \gamma_{32} & \gamma_{41} & \mid & \gamma_{24} & \gamma_{33} & \gamma_{42} & \gamma_{51}\cr
         Z \overline{Z}^2 &\gamma_{12} & \mid & \gamma_{13} & \gamma_{22}  & \mid & \gamma_{14} & \gamma_{23} & \gamma_{32} & \mid & \gamma_{15} & \gamma_{24} & \gamma_{33} & \gamma_{42}\cr
         \overline{Z}^3 &\gamma_{03} & \mid & \gamma_{04} & \gamma_{13} & \mid  & \gamma_{05} & \gamma_{14} & \gamma_{23} & \mid & \gamma_{06} & \gamma_{15} & \gamma_{24} & \gamma_{33}
     }.
\end{equation}
} Observe in passing  that each block $M[i, j]$ has a  Toeplitz
form. That is each of its  diagonals contains constant entries. On
the other hand, it is easy to see that  the matrix $M(n)$ detects
the positivity of the Riesz functional given by
$$
\Lambda_{\gamma^{(2n)}}:p(\overline{z}, z) \equiv \sum_{0\leq
i+j\leq 2n}a_{ij}\overline{z}^i z^j \longrightarrow \sum_{0\leq
i+j\leq 2n} a_{ij}\gamma_{ij}
$$
on the cone generated by the collection $\{p\overline{p}:
p\in\mathbb{C}_n[\overline{z}, z]\}$, where
$\mathbb{C}_n[Z,\overline{Z}]$ is the vector space of polynomials in
two variables with complex coefficients and total  degree less than
or equal to $n$.

It is an  immediate  observation that the rows $\overline{Z}^kZ^l$,
columns $\overline{Z}^iZ^j$ entry of the matrix $M(n)$ is equal to
$\Lambda_{\gamma^{(2n)}}(\overline{z}^{i+l}z^{j+k})
=\gamma_{i+l,j+k}$. For reason of simplicity, we identify a
polynomial $p(\overline{z}, z)\equiv \sum a_{ij} \overline{z}^i z^j$
with its coefficient vector $p = (a_{ij})$ with respect to the basis
of monomials of $\mathbb{C}_n[\overline{z}, z]$ in
degree-lexicographic order. Clearly,  $M(n)$ acts on these
coefficient vectors as follows:
\begin{equation}\label{5}
\langle M(n) p, q\rangle = \Lambda_{\gamma^{(2n)}}(p\overline{q}).
\end{equation}
A theorem of Smul'jan \cite{Shmul1959operator} shows that a block
matrix
\begin{equation}\label{qc0}
 M = \left( \begin{matrix}
 A   & B\\
 B^* & C
 \end{matrix} \right)\ge 0,
 \end{equation}
 if and only if
$$
\begin{array}{ll} (i) & A \geq 0,\\
 (ii)& \mbox{ there exists a matrix } W \mbox{  such that } B = AW, \\
 (iii) &C \geq W^*AW.
\end{array}
$$
 Since $A =A^*$, we obtain  $W^*AW$ is independent of $W$ provided that $B=AW$. Moreover, $rank~~M = rank~~A \Leftrightarrow C = W^*AW$ for some $W$ such that $B = A W$. Conversely, if $A \geq 0$, any extension $M$ satisfying $rank~~M = rank~~A$ (if this condition is satisfied, we will say  that $M$ is a flat extension of $A$) is necessarily positive. Notice  also that from the expression
\begin{equation*}
\left( \begin{matrix} I   & 0 \\ -W^*  & I'  \end{matrix} \right) M
\left( \begin{matrix} I   & -W \\ 0  & I'  \end{matrix} \right) =
\left( \begin{matrix} A   & 0 \\ 0  & C - W^*AW  \end{matrix}
\right),
\end{equation*}
where $I$ and $I'$ denote the unit matrices, we deduce that
\begin{equation}\label{C-WAW}
rank~~M = rank~~A + rank~~(C - W^*AW).
\end{equation}
By Smul'jan's theorem, $M(n) \geq 0$ admits a (necessarily positive)
flat extension
\begin{equation}\label{ext.1}
M(n+1) = \left( \begin{matrix} M(n)   & B \\ B^*  & C  \end{matrix}
\right)
\end{equation}
in the form of a moment matrix $M(n+1)$ if and only if
\begin{enumerate}
\item[$(i)$] $B = M(n) W$ for some $W$, \item[ $(ii)$] $C = W^*M(n)W$ is a Toeplitz matrix.
\end{enumerate}
We have the next result due to Curto and Fialkow,
\begin{thm}\label{c.m.t}\cite[Theorem 5.13]{CURTO1996SOLUTION}
The finite sequence  $\gamma^{(2n)}$ has a $rank~~M(n)$-atomic
representing measure if and only if $M(n) \geq 0$ and $M(n)$ admits
a flat extension $M(n+1)$. That is,  $M(n)$ can be extended to a
positive moment matrix $M(n+1)$ satisfying $rank~~M(n+1) =
rank~~M(n)$.
\end{thm}

An important step in our approach is to  show that the Hermitian
matrix $W^*M(n)W$ is persymmetric, that is, it is symmetric across
its lower-left to upper-right diagonal. For this purpose, we
introduce first some additional notation.

We denote the successive columns of $W$ and $B$ (given as in
Expression \eqref{ext.1}) by $W_{\mid Z^{n+1}}, W_{\mid \overline{Z}
Z^n}, \ldots, W_{\mid \overline{Z}^{n+1}}$ and $B_{\mid Z^{n +1}},
B_{\mid \overline{Z} Z^n}, \ldots, B_{\mid \overline{Z}^{n+1}}$ ,
respectively.

Let us consider the $\frac{(n+1)(n+2)}{2}$-matrix built as follows,
\begin{equation*}
M_{\varphi}(n) := J_0 \oplus J_1 \oplus \dots \oplus J_n;
\end{equation*}
where $J_p=(\delta_{i+j, p})_{0\leq i,j\leq p}$  with $\delta_{i,j}$
is the Kronecker symbol given by $\delta_{l,k} = 1$ for $k=l$ and
zero otherwise. For example
$$
J_0=(1), \quad J_1=\left(\begin{matrix} 0 & 1 \\ 1 & 0
\end{matrix}\right) \text{ and } J_2=\left(\begin{matrix} 0 & 0 & 1 \\ 0 & 1 &
0 \\ 1 & 0 & 0
\end{matrix}\right).
$$

\begin{lem}\label{phi} Let $M_\varphi(n)$, $M(n)$ and $B_{\overline{Z}^{n -i} Z^i}$ ($i=0, \ldots, n$) be as above, then
\begin{enumerate}
 \item $(M_\varphi(n))^2= I$. \label{1}
  \item $(M_{\varphi}(n))^*=M_{\varphi}(n)$. \label{2}
  \item $M_{\varphi}(n)B_{\overline{Z}^i Z^{n -i}} = B_{\overline{Z}^{n -i} Z^i} \qquad (i=0, \ldots, n).$ \label{4}
  \item $M_{\varphi}(n)M(n)=\overline{M(n)}M_{\varphi}(n)$. \label{3}
\end{enumerate}
\end{lem}

\begin{proof} The assertions \eqref{1}, \eqref{2} and \eqref{4} are obvious.
Only  the third assertion requires a proof. To this aim, we recall
that $M(n)=[M(i,j)]_{0\leq i,j\leq n}$, see \eqref{moment matrix}.
Therefore
\begin{eqnarray*}
  [M_{\varphi}(n)]M(n) &=& \left[\bigoplus_{i=0}^n J_i \right] [M(i,j)]_{i,j \leq n} \\
   &=& [J_i M(i,j)]_{i,j \leq n} \\
   &=& [\overline{M(i,j)} J_j]_{i,j  \leq n} \\
   &=&  [\overline{M(i,j)}]_{i,j  \leq n}\left[\bigoplus_{i=0}^n J_i
   \right]\\
   &=&\overline{M(n)}M_{\varphi}(n).
\end{eqnarray*}
\end{proof}

\begin{prop}\label{persymmetric}
 Let $n$ be a given integer and let $M(n)$ and $W$ be as above, then $W^*M(n)W$ is a Hermitian Persymmetric matrix.
\end{prop}

\begin{proof} Setting $W^*M(n)W= (c_{ij})_{0\leq i,j\leq n}$, then we have
\begin{equation}\label{per.pr.1}
c_{n-j,n-i}= W_{\overline{Z}^{n -j} Z^j}^* M(n) W_{\overline{Z}^{n
-i} Z^i} .
\end{equation}
By multiplying left  both sides of the fourth equation in Lemma
\ref{phi} by $M_{\varphi}(n)$ we obtain
\begin{equation}\label{per.pr.2}
M_{\varphi}(n) M_{\varphi}(n) M(n) = M_{\varphi}(n) \overline{M(n)}
M_{\varphi}(n).
\end{equation}
and hence, by applying Lemma \ref{phi}-(1), we have
\begin{equation}\label{per.pr.3}
M(n) = M_{\varphi}(n) \overline{M(n)} M_{\varphi}(n).
\end{equation}
It follows, from \eqref{per.pr.1} and \eqref{per.pr.3}, that
\begin{equation}
c_{n-j, n-i} = W_{\overline{Z}^{n -j} Z^j}^* M_{\varphi}(n)
\overline{M(n)} M_{\varphi}(n) W_{\overline{Z}^{n -i} Z^i}.
\end{equation}
The fact that $M_{\varphi}(n)$ is self-adjoint allows to write
\begin{equation}\label{per.pr.4}
c_{n -j, n -i} = \left( M_{\varphi}(n) W_{\overline{Z}^{n -j} Z^j}
\right)^* \overline{M(n)} \left( M_{\varphi}(n) W_{\overline{Z}^{n
-i} Z^i}\right).
\end{equation}
By using  the assertions \eqref{4} and \eqref{3}, in Lemma
\ref{phi}, we deduce that:
\begin{equation*}
\overline{M(n)} M_{\varphi}(n) W_{\overline{Z}^{n -i} Z^i} =
M_{\varphi}(n) M(n) W_{\overline{Z}^{n -i} Z^i} = M_{\varphi}(n)
B_{\overline{Z}^{n -i} Z^i}= B_{\overline{Z}^i Z^{n -i}}.
\end{equation*}
Therefore, \eqref{per.pr.4} implies that
\begin{equation*}\begin{aligned}
c_{n -j, n -i} &= ( M_{\varphi}(n) W_{\overline{Z}^{n -j} Z^j} )^*  B_{\overline{Z}^i Z^{n -i}}\\
               &= W_{\overline{Z}^{n -j} Z^j}^*  M_{\varphi}(n) M(n) W_{\overline{Z}^i Z^{n -i}}\\
               &= ((M(n) M_{\varphi}(n))^* W_{\overline{Z}^{n -j} Z^j})^* W_{\overline{Z}^i Z^{n -i}}\\
               &= (M_{\varphi}(n) M(n) W_{\overline{Z}^{n -j} Z^j})^* W_{\overline{Z}^i Z^{n -i}}\\
               &= (\overline{M(n)} M_{\varphi}(n) W_{\overline{Z}^{n -j} Z^j})^* W_{\overline{Z}^i Z^{n -i}}\\
               &= (M(n) W_{\overline{Z}^j Z^{n -j}})^* W_{\overline{Z}^i Z^{n -i}}\\
               &= W_{\overline{Z}^j Z^{n -j}}^* M(n) W_{\overline{Z}^i Z^{n -i}}\\
               &= c_{i, j}.
\end{aligned}
\end{equation*}
This concludes the proof of the Proposition \ref{persymmetric}.
\end{proof}

 \section{ Complex-valued recursive bi-sequences}
Let $\gamma^{(n)} \equiv \{\gamma_{ij}\}_{i+ j \leq n}$, with
$\overline{\gamma_{ij}} = \gamma_{ji}$ and $n \in \mathbb{N} \cup \{
+\infty\}$, be a given complex-valued sequence and let
$P_{\overline{z}^e z^{d -e}}
=\sum\limits_{\small{\begin{smallmatrix} l +k \leq d\\ (l ,k) \neq
(e, d -e) \end{smallmatrix}}} a_{lk} \overline{z}^l z^k $ be in $
\mathbb{C}_d[\overline{z}, z]$, the vector space of polynomials in
two variables with complex coefficients and total degree less than
or equal to $d$ (we assume that $d \leq n$). The sequence
$\gamma^{(n)}$ is said to be recursive, associated with a generating
polynomial $\overline{z}^e z^{d -e} -P_{\overline{z}^e z^{d -e}}$,
if
\begin{equation}\label{def.RSI}
\gamma_{e +i, d -e +j} = \Lambda_{\gamma^{(n)}} (\overline{z}^i z^j
P_{\overline{z}^e z^{d -e}}), \phantom{tttt} \text{ for all } i +j
\leq n -d,
\end{equation}
or, equivalently, if
\begin{equation}\label{def.RS-II}
\gamma_{e +i, d -e +j} = \sum\limits_{\begin{smallmatrix} l +k \leq
d\\ (l ,k) \neq (e, d -e) \end{smallmatrix}} a_{lk} \gamma_{l +i, k
+j} \phantom{ttttttt} (i +j \leq n -d).
\end{equation}

We notice that, because of the equality $\overline{\gamma_{ij}} =
\gamma_{ji}$, Equation \eqref{def.RS-II} is equivalent to the
following one:
\begin{equation}\label{def.RS-III}
\gamma_{d -e +i, e +j} = \sum\limits_{\begin{smallmatrix} l +k \leq
d\\ (l ,k) \neq (e, d -e) \end{smallmatrix}} \overline{a_{lk}}
\gamma_{k +i, l +j},
\end{equation} for all integers $i$ and $j$, with $i + j \leq n -d$,\\
Therefore, $\overline{z}^{d -e} z^e -P_{\overline{z}^{d -e} z^e}$
(where $P_{\overline{z}^{d -e} z^e} = \overline{P_{\overline{z}^e
z^{d -e}}}$) is, also, a generating polynomial, associated with
$\gamma^{(n)}$; that is,
\begin{equation}\label{def.RS-IV}
\gamma_{d -e +i, e +j} = \Lambda_{\gamma^{(n)}} (\overline{z}^i z^j
P_{\overline{z}^{d -e} z^e}), \phantom{tttttttttttt} i +j \leq n -d.
\end{equation}

The following proposition provides a connection, via $\Lambda$,
between the polynomials  $P_{\overline{z}^f z^{f+1}}$ and
$P_{\overline{z}^{f+1} z^f}$.
\begin{prop}\label{prop.lien}
Let $\gamma^{(n)} \equiv \{ \gamma_{ij} \}_{i+j \leq n}$ be a
recursive bi-sequence and let $\overline{z}^f z^{f+1}
-P_{\overline{z}^f z^{f+1}}$ be an  associated generating
polynomial, then
\begin{equation*}
\Lambda_{\gamma^{(n)}}(\overline{z}^{l+1} z^k P_{\overline{z}^f
z^{f+1}}) = \Lambda_{\gamma^{(n)}}(\overline{z}^l z^{k+1}
P_{\overline{z}^{f+1} z^f}), \phantom{tttttt} l +k \leq n -2f -2.
\end{equation*}
\end{prop}

\begin{proof}
 For all integers $l$ and $k$, with $l +k \leq n -2f -2$, we have
\begin{equation*}\begin{aligned}
\Lambda_{\gamma^{(n)}}(\overline{z}^{l+1} z^k P_{\overline{z}^f z^{f+1}}) &= \gamma_{f+l+1, f+k+1}\\
                                                            &= \overline{\gamma_{f+k+1, f+l+1}}\\
                                                            &= \overline{\Lambda_{\gamma^{(n)}}(\overline{z}^{k+1} z^l P_{\overline{z}^f z^{f+1}})}\\
                                                            &= \Lambda_{\gamma^{(n)}}(\overline{z}^l z^{k+1} P_{\overline{z}^{f+1} z^f}).
\end{aligned}
\end{equation*}
\end{proof}

It is well known that the (classical singly  indexed recursive
sequence can be defined by the initial data and the, associated
recurrence relation (or, characteristic polynomial), see
\cite{dubeau1994weighted}.  In a similar way, one can define
recursive bi-sequences   as observed below.

\begin{rem}\label{rem.construct.}
\begin{itemize}
\item [$i)$]  A generating polynomial $z^e -P_{z^e}$ (or, equivalently, $\overline{z}^e -P_{\overline{z}^e}$),
with $\deg P_{z^e} < e$, together with the initial data
$\{\gamma_{ij}\}_{i, j < e}$ and the equality
$\overline{\gamma_{ij}} = \gamma_{ji}$, are said to define the
sequence $\gamma^{(n)}$.
\item [$ii)$] For  a generating polynomial $\overline{z} z^{e -1} -P_{\overline{z} z^{e -1}}$,
with $\deg P_{\overline{z} z^{e -1}} < e$, we need (all) the data
$\{\gamma_{ij}\}_{i, j < e} \cup \{ \gamma_{0j} \}_{j = e, \ldots,
n}$ and the equality $\overline{\gamma_{ij}} = \gamma_{ji}$ to
define  the recursive bi-sequence $\gamma^{(n)}$.
\end{itemize}
\end{rem}

In the next lemmas, we provide useful results for solving the
quintic moment problem.

\begin{lem}\label{lem.II-1}
Let $\gamma^{(8)} \equiv \{\gamma_{ij}\}_{i+ j \leq 8}$, with
$\overline{\gamma_{ij}} = \gamma_{ji}$, be a truncated bi-sequence
and let $z^4 -P_{z^4}$ (where $P_{z^4} = \beta z^3 + R_{z^4}$ and
$R_{z^4} \in \mathbb{C}_2[\overline{z}, z]$) be an associated
generating polynomial. Assume that $\overline{z} z^2-P_{\overline{z}
z^2}$ (where $P_{\overline{z} z^2} = \alpha z^3 + R_{\overline{z}
z^2}$, $\alpha \neq 0$ and $R_{\overline{z} z^2} \in
\mathbb{C}_2[\overline{z}, z]$) is a generating polynomial for
$\gamma^{(6)} \cup \{ \gamma_{34}, \gamma_{43}\}$, then
$\overline{z} z^2 -P_{\overline{z} z^2}$ is a generating polynomial
for $\gamma^{(8)}$.
\end{lem}

\begin{proof} We have $z^4 -P_{z^4}$ is  a generating polynomial for
$\gamma^{(8)}$, that is,
\begin{equation}\label{pr.1;lem.II-1}
\gamma_{i ,j +4} = \Lambda_{\gamma^{(8)}} (\overline{z}^i z^j
P_{z^4}) = \beta \gamma_{i, j +3} +\Lambda_{\gamma^{(8)}}
(\overline{z}^i z^j R_{z^4}), \phantom{ttttt} i +j \leq 4.
\end{equation}
As showing in \eqref{def.RS-IV}, the last equality
\eqref{pr.2;lem.II-1} is equivalent to
\begin{equation}\label{pr.1-1;lem.II-1}
\gamma_{i +4, j} = \Lambda_{\gamma^{(8)}} (\overline{z}^i z^j
P_{\overline{z}^4}) = \overline{\beta} \gamma_{i +3, j}
+\Lambda_{\gamma^{(8)}} (\overline{z}^i z^j R_{\overline{z}^4}),
\phantom{ttttt} i +j \leq 4;
\end{equation}
where $\overline{P_{z^4}} := P_{\overline{z}^4} = \overline{\beta} \overline{z}^3 + \overline{R_{\overline{z} z^2}}$.\\
Also, the polynomial $\overline{z} z^2 -P_{\overline{z} z^2}$ is a
generating one for $\gamma^{(6)} \cup \{ \gamma_{34},
\gamma_{43}\}$; that is, for all $i +j \leq 3$  and $(i, j) = (2, 2)
, (3, 1)$:
\begin{equation}\label{pr.2;lem.II-1}
\gamma_{i +1 ,j +2} = \Lambda_{\gamma^{(8)}} (\overline{z}^i z^j
P_{\overline{z} z^2}) = \alpha \gamma_{i ,j +3}
+\Lambda_{\gamma^{(8)}} (\overline{z}^i z^j R_{\overline{z} z^2}).
\end{equation}
or, equivalently, for $i +j \leq 3$ and  $(i, j) = (2, 2) , (1, 3)$;
\begin{equation}\label{pr.2-2;lem.II-1}
\gamma_{i +2 ,j +1} = \Lambda_{\gamma^{(8)}} (\overline{z}^i z^j
P_{\overline{z}^2 z^1}) = \overline{\alpha} \gamma_{i +3, j}
+\Lambda_{\gamma^{(8)}} (\overline{z}^i z^j R_{\overline{z}^2 z}),
\end{equation}
where $P_{\overline{z}^2 z} := \overline{P_{\overline{z} z^2}} = \overline{\beta} \overline{z}^3 + R_{\overline{z}^2 z}$, see \eqref{def.RS-IV}.\\
We have to show that \eqref{pr.2;lem.II-1} remains valid for all
integers $i$ and $j$, with $i +j \leq 5$. To this end we consider
the recursive bi-sequence $\widehat{\gamma}^{(8)} \equiv
\{\widehat{\gamma}_{ij} \}_{i+j \leq 8}$ defined by
  \begin{equation}\label{pr.3;lem.II-1}
  \begin{cases}
  \widehat{\gamma}_{i +1, j +2} &= \Lambda_{\widehat{\gamma}^{(8)}}( \overline{z}^i z^j P_{\overline{z} z^2}) \qquad ( i +j \leq 5),\\
  \widehat{\gamma}_{i, j} &= \gamma_{i, j} \qquad \textrm{ otherwise
  };
  \end{cases}
  \end{equation}
and we will show that $\widehat{\gamma}^{(8)} = \gamma^{(8)}$.
Notice that since $\overline{z} z^2 -P_{\overline{z} z^2}$ is a
generating polynomial for $\widehat{\gamma}^{(8)}$, then
$\overline{z}^2 z -P_{\overline{z}^2 z}$ is an other one. Thus
\begin{equation}\label{pr.33;lem.II-1}
  \widehat{\gamma}_{i +2, j +1} = \Lambda_{\widehat{\gamma}^{(8)}}( \overline{z}^i z^j P_{\overline{z}^2 z}) \phantom{texta} ( i +j \leq 5).
\end{equation}
It follows from \eqref{pr.2;lem.II-1} and \eqref{pr.3;lem.II-1}
that, for $n +m \leq 6$, $n = 0$ and $(n, m) = (3, 4) , (4, 3)$:
\begin{equation}\label{pr.4;lem.II-1}
\gamma_{nm} = \Lambda_{\gamma^{(8)}}(\overline{z}^n z^m) =
\Lambda_{\widehat{\gamma}^{(8)}}(\overline{z}^n z^m) :=
\widehat{\gamma}_{nm}.
\end{equation}
Remark that if $\widehat{\gamma}_{nm} = \gamma_{nm}$ then
$\widehat{\gamma}_{mn} = \overline{\widehat{\gamma}_{nm}} =
\overline{\gamma_{nm}} = \gamma_{mn}$.

Therefore, we need to show \eqref{pr.4;lem.II-1}, only, for the
integers $n$ and $m$ with $(n, m) = (2, 5), (1, 6); (1, 7), (2, 6),
(3, 5), (4, 4)$.

\begin{equation}\label{e(25)}\begin{aligned}
\gamma_{25} &= \Lambda_{\gamma^{(8)}}(\overline{z}^2 z P_{z^4}), &\text{utilizing } \eqref{pr.1;lem.II-1},\\
&= \Lambda_{\gamma^{(8)}}(P_{\overline{z}^2 z} P_{z^4}), &\text{employ } \eqref{pr.2-2;lem.II-1} \text{ and } \deg P_{z^4} \leq 3,\\
&= \overline{\alpha} \Lambda_{\gamma^{(8)}}(\overline{z}^3 P_{z^4}) + \Lambda_{\gamma^{(8)}}(R_{\overline{z}^2 z} P_{z^4})\\
&= \overline{\alpha} \gamma_{34} + \Lambda_{\gamma^{(8)}}( z^4 R_{\overline{z}^2 z}), &\text{ applying } \eqref{pr.1;lem.II-1},\\
&= \overline{\alpha} \widehat{\gamma}_{34} + \Lambda_{\widehat{\gamma}^{(8)}} ( z^4 R_{\overline{z}^2 z}), &\text{use } \deg z^4 R_{\overline{z}^2 z} \leq 6 \text{ and } \eqref{pr.4;lem.II-1},\\
&= \Lambda_{\widehat{\gamma}^{(8)}} (\overline{\alpha} \overline{z}^3 z^4 + z^4 R_{\overline{z}^2 z})\\
&= \Lambda_{\widehat{\gamma}^{(8)}} (z^4 P_{\overline{z}^2 z})\\
&= \Lambda_{\widehat{\gamma}^{(8)}} (\overline{z} z^3 P_{z^2\overline{z}}), &\text{ according to Proposition \ref{prop.lien}},\\
&= \widehat{\gamma}_{25}, &\text{ from } \eqref{pr.33;lem.II-1}.
  \end{aligned}
  \end{equation}

\begin{equation}\label{e(16)}\begin{aligned}
  \gamma_{16} &= \Lambda_{\gamma^{(8)}}(\overline{z} z^2 P_{z^4}), &\text{use} \eqref{pr.1;lem.II-1},\\
              &= \Lambda_{\gamma^{(8)}}(P_{\overline{z} z^2} P_{z^4}), &\text{ employ } \eqref{pr.2;lem.II-1} \text{ and } \deg P_{z^4} \leq 3,\\
              &= \Lambda_{\gamma^{(8)}}(\alpha z^3 P_{z^4} + R_{\overline{z} z^2} P_{z^4})\\
              &= \alpha \gamma_{07} + \Lambda_{\gamma^{(8)}}( z^4 R_{\overline{z} z^2}), &\text{ utilizing } \eqref{pr.1;lem.II-1},\\
              &= \alpha \widehat{\gamma}_{07} + \Lambda_{\widehat{\gamma}^{(8)}} ( z^4 R_{\overline{z} z^2}), &\text{ using } \eqref{pr.4;lem.II-1} \text{ and } \deg z^4 R_{\overline{z} z^2} \leq 6,\\
              &= \Lambda_{\widehat{\gamma}^{(8)}} (\alpha z^7 + z^4 R_{\overline{z} z^2})\\
              &= \Lambda_{\widehat{\gamma}^{(8)}} (z^4 P_{\overline{z} z^2})\\
              &= \widehat{\gamma}_{16}, &\text{ according to } \eqref{pr.3;lem.II-1}.
  \end{aligned}
  \end{equation}
  Thus, the equality \eqref{pr.4;lem.II-1} is valid for every integer $n$ and $m$ with $n +m \leq 7$. In other words,
  \begin{equation}\label{pr.44;lem.II-1}
\gamma_{nm} = \Lambda_{\gamma^{(8)}}(\overline{z}^n z^m) =
\Lambda_{\widehat{\gamma}^{(8)}}(\overline{z}^n z^m) :=
\widehat{\gamma}_{nm} \phantom{tttttttttttt}
 (n +m \leq 7).
\end{equation}
And thus one can generalize the relation \eqref{pr.2;lem.II-1} as
follows
\begin{equation}\label{pr.5;lem.II-1}
\gamma_{i +1 ,j +2} = \Lambda_{\gamma^{(8)}} (\overline{z}^i z^j
P_{\overline{z} z^2}) = \alpha \gamma_{i ,j +3}
+\Lambda_{\gamma^{(8)}} (\overline{z}^i z^j R_{\overline{z} z^2})
\phantom{ttttt} ( i +j \leq 4).
\end{equation}

  Now, let us show \eqref{pr.4;lem.II-1} in the remaining cases ($n + m = 8$).
  \begin{equation}\label{e(08)}
   \gamma_{08} = \widehat{\gamma}_{08}, \phantom{tttt} \text{ by the construction of $\widehat{\gamma}^{(8)}$, see \eqref{pr.3;lem.II-1}}.
   \end{equation}

   \begin{equation}\label{e(17)}\begin{aligned}
  \gamma_{17} &= \Lambda_{\gamma^{(8)}}(\overline{z} z^3 P_{z^4}), &\text{ according to } \eqref{pr.1;lem.II-1}\\
              &= \Lambda_{\gamma^{(8)}}(z P_{z^4} P_{\overline{z} z^2}), &\text{ because } \deg z P_{z^4} \leq 4 \text{ and } \eqref{pr.5;lem.II-1},\\
              &= \Lambda_{\gamma^{(8)}}(z^5 P_{\overline{z} z^2}), &\text{ utilizing } \eqref{pr.1;lem.II-1}\\
              &= \alpha \Lambda_{\gamma^{(8)}}(z^8) + \Lambda_{\gamma^{(8)}}( z^5 R_{\overline{z} z^2})\\
              &= \alpha \Lambda_{\widehat{\gamma}^{(8)}} (z^8) + \Lambda_{\widehat{\gamma}^{(8)}} ( z^5 R_{\overline{z} z^2}), &\text{ using } \eqref{e(08)} \text{ and } \eqref{pr.44;lem.II-1},\\
              &= \Lambda_{\widehat{\gamma}^{(8)}} (z^5 P_{\overline{z} z^2})\\
              &= \widehat{\gamma}_{17}, &\text{ applying } \eqref{pr.3;lem.II-1}.
  \end{aligned}
  \end{equation}

  \begin{equation}\label{e(26)}\begin{aligned}
  \gamma_{26} &= \Lambda_{\gamma^{(8)}}(\overline{z}^2 z^2 P_{z^4}), &\text{ according to } \eqref{pr.1;lem.II-1}\\
              &= \Lambda_{\gamma^{(8)}}(\overline{z} P_{z^4} P_{\overline{z} z^2}), &\text{ use } \deg \overline{z} P_{z^4} \leq 4 \text{ and } \eqref{pr.5;lem.II-1},\\
              &= \Lambda_{\gamma^{(8)}}(\overline{z} z^4 P_{\overline{z} z^2}), &\text{ utilizing } \eqref{pr.1;lem.II-1},\\
              &= \alpha \Lambda_{\gamma^{(8)}}(\overline{z} z^7) + \Lambda_{\gamma^{(8)}}( \overline{z} z^4 R_{\overline{z} z^2})\\
              &= \alpha \Lambda_{\widehat{\gamma}^{(8)}} (\overline{z} z^7) + \Lambda_{\widehat{\gamma}^{(8)}} ( \overline{z} z^4 R_{\overline{z} z^2}),&\text{ by using } \eqref{e(17)} \text{ and } \eqref{pr.44;lem.II-1},\\
              &= \Lambda_{\widehat{\gamma}^{(8)}} (\overline{z} z^4 P_{\overline{z} z^2})\\
              &= \widehat{\gamma}_{26}, &\text{ according to } \eqref{pr.3;lem.II-1}.
  \end{aligned}
  \end{equation}

  Before continue the proof, of these lemma, let us remark that the Relation \ref{pr.44;lem.II-1} implies that, for all $i +j \leq 5$,
  \begin{equation*}
\Lambda_{\widehat{\gamma}^{(8)}}(\overline{z}^{i +1} z^{j +2}) =
\widehat{\gamma}_{i +1, j +2} = \Lambda_{\widehat{\gamma}^{(8)}}(
\overline{z}^i z^j (\alpha z^3 +R_{\overline{z} z^2})),
  \end{equation*}
and thus
\begin{equation}\label{pr.444;lem.II-1}
\Lambda_{\widehat{\gamma}^{(8)}}(\overline{z}^{i} z^{j +3})
=\frac{1}{\alpha} \Lambda_{\widehat{\gamma}^{(8)}}(\overline{z}^{i}
z^{j}(\overline{z} z^2 -R_{\overline{z} z^2})) \phantom{ttttttttt}
(i+j \leq 5).
  \end{equation}
  Now,
  \begin{equation}\label{e(35)}\begin{aligned}
  \gamma_{35} &= \Lambda_{\gamma^{(8)}}(\overline{z}^3 z P_{z^4}), &\text{ according to } \eqref{pr.1;lem.II-1},\\
              &= \Lambda_{\widehat{\gamma}^{(8)}}(\overline{z}^3 z P_{z^4}), &\text{ because } \deg \overline{z}^3 z P_{z^4} \leq 7,\\
              &= \Lambda_{\widehat{\gamma}^{(8)}}(\frac{1}{\alpha}(P_{\overline{z}^2 z} -R_{\overline{z}^2 z}) z P_{z^4})\\
              &= \frac{1}{\alpha} \Lambda_{\widehat{\gamma}^{(8)}}((\overline{z}^2 z -R_{\overline{z}^2 z}) z P_{z^4})\\
              &= \frac{1}{\alpha} \Lambda_{\gamma^{(8)}}((\overline{z}^2 z -R_{\overline{z}^2 z}) z P_{z^4}), &\text{applying } \eqref{pr.44;lem.II-1},\\
              &= \frac{1}{\alpha} \Lambda_{\gamma^{(8)}}(\overline{z}^2 z^6) -\frac{1}{\alpha} \Lambda_{\gamma^{(8)}}( z^5 R_{\overline{z}^2 z})\\
              &= \frac{1}{\alpha} \gamma_{26} -\frac{1}{\alpha} \Lambda_{\widehat{\gamma}^{(8)}}( z^5 R_{\overline{z}^2 z}),& \text{ remark that } \deg z^5 R_{\overline{z}^2} \leq 7,\\
              &= \frac{1}{\alpha} \widehat{\gamma}_{26} -\frac{1}{\alpha} \Lambda_{\widehat{\gamma}^{(8)}}(z^5 R_{\overline{z}^2 z}),&\text{from } \eqref{e(26)},\\
              &=\Lambda_{\widehat{\gamma}^{(8)}}(\frac{1}{\alpha}(P_{\overline{z}^2 z} -R_{\overline{z}^2 z}) z^5)\\
              &=\Lambda_{\widehat{\gamma}^{(8)}}(\overline{z}^3 z^5)\\
              &=\widehat{\gamma}_{35}.
  \end{aligned}
  \end{equation}

  \begin{equation}\label{e(44)}\begin{aligned}
\gamma_{44} &= \Lambda_{\gamma^{(8)}}(\overline{z}^4 P_{z^4})\\
            &= \Lambda_{\widehat{\gamma}^{(8)}}(\overline{z}^3 \overline{z} P_{z^4}) &\text{using } \eqref{pr.44;lem.II-1},\\
            &= \Lambda_{\widehat{\gamma}^{(8)}}(\frac{1}{\alpha}(P_{\overline{z}^2 z} -R_{\overline{z}^2 z}) \overline{z} P_{z^4}),&\text{by } \eqref{pr.444;lem.II-1},\\
            &= \frac{1}{\alpha} \Lambda_{\widehat{\gamma}^{(8)}}(\overline{z}^3 zP_{z^4}) -\frac{1}{\alpha} \Lambda_{\widehat{\gamma}^{(8)}}(\overline{z} P_{z^4} R_{\overline{z}^2 z})\\
            &= \frac{1}{\alpha} \Lambda_{\gamma^{(8)}}(\overline{z}^3 zP_{z^4}) -\frac{1}{\alpha} \Lambda_{\gamma^{(8)}}(\overline{z} P_{z^4} R_{\overline{z}^2 z})\\
            &= \frac{1}{\alpha}\gamma_{35} -\frac{1}{\alpha} \Lambda_{\widehat{\gamma}^{(8)}}(\overline{z} P_{z^4} R_{\overline{z}^2 z})\\
            &= \frac{1}{\alpha}\widehat{\gamma}_{35} -\frac{1}{\alpha} \Lambda_{\widehat{\gamma}^{(8)}}(\overline{z} P_{z^4} R_{\overline{z}^2 z}), &\text{applying } \eqref{e(35)},\\
            &= \Lambda_{\widehat{\gamma}^{(8)}}(\frac{1}{\alpha}(P_{\overline{z}^2 z} -R_{\overline{z}^2 z}) \overline{z} P_{z^4})\\
            &= \Lambda_{\widehat{\gamma}^{(8)}}(\overline{z}^4 P_{z^4})\\
            &= \widehat{\gamma}_{44}.
\end{aligned}
\end{equation}
This finishes the proof of Lemma \ref{lem.II-1}.
\end{proof}

\section{ Solving the quintic moment problem}

Let $\gamma^{(5)} \equiv \{\gamma_{ij}\}_{i+j\leq 5}$ be a given
complex-valued bi-sequence, with $\gamma_{00} >0$ and
 $\overline{\gamma_{ij}} = \gamma_{ji}$
for $i+j\leq 5$. The quintic moment problem involves determining
necessary and sufficient conditions for the existence of a positive
Borel measure $\mu$ on
 $\mathbb{C}$ (called a representing measure for $\gamma^{(5)}$) such that

  \begin{equation*}
  \gamma_{ij} = \int \overline{z}^i z^j d\mu, \phantom{text} \text{ for } i +j \leq 5.
  \end{equation*}

In this section we will show that in almost all cases the classical
necessary conditions $M(2) \geq 0$ and $B = M(2) W$, for some $W$,
(with $M(2)$ and $B$ are as in \eqref{ext.1}) guarantee the
existence of at most $(r+2)$-atomic (here $r := rank~~M(2)$)
representing measure for $\gamma^{(5)}$.

According to Proposition \ref{persymmetric}, the Hermitian $4 \times
4$-matrix  $W^* M(2) W$ is symmetric with respect to the second
diagonal, then one
 can set
 \begin{equation}\label{qc1}
 W^* M(2) W = \left( \begin{matrix}
 a            & b            & c & d\\
 \overline{b} & e            & f & c\\
 \overline{c} & \overline{f} & e & b\\
 \overline{d} & \overline{c} & \overline{b}     & a
 \end{matrix} \right)
 \end{equation}

    The next Theorem gives a concrete solution to the quintic complex moment
    problem, except for the case $a=e$ and $b\neq f$.

    \begin{thm}\label{mainlast}
     Let $\gamma^{(5)} \equiv \{\gamma_{ij}\}_{i+j\leq 5}$ be a given sequence, we assume that $M(2) \geq 0$  and $Rang~~B \subseteq Rang~~M(2)$, and $a\neq e$ or $b=f$

Then the quintic moment problem, associated with  $\gamma^{(5)}$,
admits a solution $\mu$. Moreover, The smallest cardinality of
$supp~~\mu$ is
  \begin{itemize}
    \item  $card~~supp~~\mu = r \iff a = e$ and $b = f$,
    \item  $card~~supp~~\mu = r +1 \iff  \quad a\neq e \textrm{ and } \frac{a-e}{2} < \mid b -f
    \mid$,
    \item  $card~~supp~~\mu = r +2 \iff \quad a>e \textrm{ and } \frac{a-e}{2} \geq \mid b -f
    \mid$;
  \end{itemize}
  where $a, b, e$ and $f$ are as in  \eqref{qc1}.
     \end{thm}
Before we develop the proof of our  theorem, let us  introduce some
notations. For  $n \in \{3, 4\}$; let $\gamma^{(2n)} \equiv
\{\gamma_{ij}\}_{i+j\leq 2n}$ be a truncated complex bi-sequence and
let $M(n)$ be the associated moment matrix. As before, we denote by
$B(n)$ and $C(n)$ the $(n-1)\times n$-matrix and the $n \times
n$-matrix, respectively, such that
  \begin{equation}\label{mr0}
  M(n) = \left(\begin{matrix}
  M(n -1) & B(n) \\
  B^*(n)  & C(n)
  \end{matrix} \right)
  \end{equation}

    Let $\mathfrak{B} \equiv \mathfrak{B}(n) \equiv \{ \overline{Z}^i Z^j \}_{(i, j) \in \mathfrak{R}}$
  (where $\mathfrak{R} \equiv \mathfrak{R}(n) \subseteq \{0, 1, \ldots, n\} \times \{0, 1, \ldots, n\}$) be a basis for the column space of $M(n)$.
Let us remark that the  $ r \times r$-matrix  $M(n)_{\mid
\mathfrak{B}}$, where $r \equiv r(n) := card~~\mathfrak{R}(n)$, the
restriction of the moment matrix $M(n)$ to the basis $\mathfrak{B}$,
is invertible.

%


{\it Proof of Theorem \ref{mainlast}.} The main  idea is to extend
the initial data $\gamma^{(5)}$ to an even-degree $\gamma^{(6)}$ (by
adding the sixtic moments $\gamma_{60} = \overline{\gamma_{06}}$,
$\gamma_{51} = \overline{\gamma_{15}}$, $\gamma_{42} =
\overline{\gamma_{24}}$ and $\gamma_{33} \in \mathbb{R}$)  such that
the associated moment matrix $M(3)$, for an appropriate choice of
the missing  moments, is either a flat extension of $M(2)$ or admits
admits a flat extension $M(4)$. Thus Theorem \ref{c.m.t} yields that
$M(3)$ has a representing measure; and  as a consequence,
$\gamma^{(5)}$ also admits a representing measure $\mu$. It is also
proved that the smallest cardinality of $supp~~\mu$ will be $r :=
rank~~M(2)$ or $r+1$ or $r+2$.

By virtue of the Smul'jan's Theorem, we need to find a Toeplitz
square matrix $C(3)$, built with  the new, sixtic, moments as
entries and  such that $C(3) - W^* M(2) W \geq 0$. Setting
   \begin{equation}\label{qc2}
 C(3) - W^* M(2) W = \left( \begin{matrix}
 \gamma_{33} -a            & \gamma_{42} -b            & \gamma_{51} -c & \gamma_{60} -d\\
 \gamma_{24} -\overline{b} & \gamma_{33} -e            & \gamma_{42} -f & \gamma_{51} -c\\
 \gamma_{15} -\overline{c} & \gamma_{24} -\overline{f} & \gamma_{33} -e & \gamma_{42} -b\\
 \gamma_{06} -\overline{d} & \gamma_{15} -\overline{c} & \gamma_{24} -\overline{b}     & \gamma_{33} -a
 \end{matrix} \right),
 \end{equation}
we will distinguish two cases:\\

{\bf  Case I: $a = e$ and $b = f$. }  In this case the matrix $W^*
M(2) W$ is a Toeplitz one, then it suffice to consider that $C(3) =
W^* M(2) W$. According to \eqref{C-WAW}, the matrix $M(3)$ is a flat
extension of $M(2)$ and thus $\gamma^{(6)}$ (and in force
$\gamma^{(5)}$) has a $r$-representing measure.

{\bf Case II: $a \neq e$.}   We proceed in two steps  for this case.
 obviously,  the matrix $W^* M(2) W$ is not a Toeplitz one.
 Therefore, for every choice of a Toeplitz $4\times 4$-matrix $C(3)$, we have $rank~~(C(3) - W^* M(2) W) \geq 1$.
  We will show, in first step, that the smallest possible $rank$ of $C(3) - W^* M(2) W$ will be  either $1$  or $2$.
   In the second step, we will show that the moment matrix $M(3)$, obtained by extending  $\gamma^{(5)}$ with the entries
   of some suitable $C(3)$, has a flat extension and thus admits a $rank~~M(3)$-atomic representing measure, see Theorem \ref{c.m.t}.

 {\it Step 1:  (construction of $C(3)$).} Firstly, let us observe that
 \begin{equation}\label{pr.0}
  rank (C(3) - W^* M(2) W) = 1 \textrm{ and } C(3) - W^* M(2) W\geq 0
  \end{equation}
  if and only if we have
 \begin{equation}\label{pr.1}\begin{aligned}
 (0)   & \; \gamma_{33}> max(a,e).\\
 (i)   & \mid \gamma_{42} -b \mid = \sqrt{(\gamma_{33} -a)(\gamma_{33} -e)} &\text{and} &  \mid \gamma_{42} -f \mid = \gamma_{33} -e.\\
 (ii)  & (\gamma_{15} -\overline{c}) ( \gamma_{42} -b )= (\gamma_{33} -a) (\gamma_{24} -\overline{f}).\\
 (iii) & (\gamma_{06} -\overline{d}) (\gamma_{42} -b)^2= (\gamma_{33} -a)^2 (\gamma_{24} -\overline{f}) &\text{and} & \mid \gamma_{06} -\overline{d} \mid^2 = (\gamma_{33} -a)^2.
 \end{aligned}
 \end{equation}
 Remark that the equalities $(i)$ provide the compatibility of the two equalities in $(iii)$ and vice versa.

The condition $(i)$ means that $\gamma_{42}$ is in the intersection
of the two next circles $\mathcal{C} = \mathcal{C}(b,
\sqrt{(\gamma_{33} -a)(\gamma_{33} -e)})$, of radius
$\sqrt{(\gamma_{33} -a)(\gamma_{33} -e)}$ and centered at  $b$, and
$\mathcal{C'}=\mathcal{C}(f, \gamma_{33} -e)$.

It is an easy geometrical observation to see that, the two circles
$\mathcal{C}$ and  $\mathcal{C'}$ have a nonempty intersection if,
and only if, there exists $\gamma_{33} > \max(a, e)$, such that
 \begin{equation}\label{pr.3}
 \mid (\gamma_{33} -e) - \sqrt{(\gamma_{33} -a)(\gamma_{33} -e)} \mid \leq \mid b -f \mid \leq \gamma_{33} -e + \sqrt{(\gamma_{33} -a)(\gamma_{33} -e)}
 \end{equation}

 As $x \rightarrow (x -e) - \sqrt{(x -a)(x -e)}$ is decreasing (on $[\max(a, e); +\infty[$),
 $ (x -e) - \sqrt{(x -a)(x -e)} \xrightarrow[x \rightarrow +\infty]{} \frac{a-e}{2}$   and $(x -e) + \sqrt{(x -a)(x -e)} \xrightarrow[x \rightarrow +\infty]{} +\infty.$
 Then \eqref{pr.3} is verified if and only if $a=e$ and $b\neq f$ or $a<e$ or $a>e$ and $|b-f|>\frac{a-e}{2}$.

 {\bf Subcase II-1: $a<e$ or $a>e$ and $|b-f|>\frac{a-e}{2}$.}
 It suffices to choose $\gamma_{33}$ verifying \eqref{pr.3}, and thus $\gamma_{42}$ exists (as the point intersection of the two circles $\mathcal{C}$ and $\mathcal{C}'$).
  Furthermore , from $(0)$ and $(i)$ we derive that
\begin{equation}\label{alpha}
(\gamma_{42} -b)(\gamma_{42} -f) \neq 0
\end{equation}
 The equality $(ii)$ gives the moment $\gamma_{15}$ and $(iii)$ supplies $\gamma_{06}$, and this complete the construction
 of a Toeplitz matrix $C(3)$ for which $rank (C(3) - W^* M(2) W) = 1$. Note that, $rank~~M(3)_{\mid \mathfrak{B}(2) \cup \{ Z^3 \} } = rank~~M(2) +1 = rank~~M(3)$.
  Hence, in $M(3)$, the columns $\overline{Z} Z^2, \overline{Z}^2 Z$
and $\overline{Z}^3$ are a linear combination of the columns
$\mathfrak{B}(2) \cup \{ Z^3 \}$. In particular, we can set
 \begin{equation} \label{line12}
 \overline{Z} Z^2 = P_{\overline{z} z^2} (Z, \overline{Z}) = \alpha Z^3 + R_{\overline{z} z^2} (Z, \overline{Z}),
 \end{equation}
 with
  \begin{equation}\label{alpha2}\begin{aligned}
 \alpha &= \frac{ \det \begin{vmatrix} M(2)_{\mid \mathfrak{B}(2) } & {\overline{Z}Z^2}_{\mid \mathfrak{B}(2) } \\ (Z^3_{\mid \mathfrak{B}(2) })^* & \gamma_{42} \end{vmatrix}}
 {\det \begin{vmatrix} M(2)_{\mid \mathfrak{B}(2) } & Z^3_{\mid \mathfrak{B}(2) }\\ (Z^3_{\mid \mathfrak{B}(2) })^* & \gamma_{33} \end{vmatrix}}\\
   &= \frac{ \det \begin{vmatrix} M(2)_{\mid \mathfrak{B}(2) } & {\overline{Z}Z^2}_{\mid \mathfrak{B}(2) } \\ (Z^3_{\mid \mathfrak{B}(2) })^* &  b \end{vmatrix} + (\gamma_{42} -b)~~~~\det\mid M(2)_{\mid \mathfrak{B}(2)}\mid }
 {\det \begin{vmatrix} M(2)_{\mid \mathfrak{B}(2) } & Z^3_{\mid \mathfrak{B}(2) }\\ (Z^3_{\mid \mathfrak{B}(2) })^* & a \end{vmatrix} +(\gamma_{33} -a)~~~~\det \mid M(2)_{\mid \mathfrak{B}(2)}\mid}\\
  &= \frac{(\gamma_{42} -b)~~~~\det \mid M(2)_{\mid \mathfrak{B}(2)}\mid}{(\gamma_{33} -a)~~~~\det \mid M(2)_{\mid \mathfrak{B}(2)} \mid}\\
  &= \frac{\gamma_{42} -b}{\gamma_{33} -a} \neq 0; ~~~~ \text{ by virtue of } \eqref{alpha}.
    \end{aligned}\end{equation}

  {\bf Subcase II-2:  $a>e$ and $\frac{a -e}{2} \geq \mid b -f \mid$.} Then $rank (C(3) - W^* M(2) W) \geq 2$ for every $4 \times 4$-Toeplitz matrix $C(3)$. Let us choose the sixtic moments as follows
  \begin{equation}\label{pr.4}\begin{cases}
  \gamma_{33} > \max(a, e),\\
  \mid \gamma_{42} -b \mid = \sqrt{(\gamma_{33} -a)(\gamma_{33} -e)},\\
  \gamma_{15} -\overline{c} = \frac{\gamma_{33} -a}{\gamma_{42} -b} (\gamma_{24} -\overline{f})\\
  \gamma_{06} -\overline{d} = (\frac{\gamma_{33} -a}{\gamma_{42} -b})^2 (\gamma_{24} -\overline{f}).
  \end{cases}
  \end{equation}
  Let us remark that as the first subcase II-1, we have
    \begin{equation}\label{alpha3}
  \frac{\gamma_{42} -b}{\gamma_{33} -a} = \sqrt{\frac{\gamma_{33} -e}{\gamma_{33} -a}} \neq 0.
  \end{equation}
  The moment defined in \eqref{pr.4} construct a Toeplitz matrix $C(3)$ for which $rank (C(3) - W^* M(2) W) = 2$. Indeed, it suffices to observe that
  \begin{itemize}
    \item $(C(3) - W^* M(2) W)(Z^3) =  \frac{\gamma_{33} -a}{\gamma_{42} -b} (C(3) - W^* M(2) W)(\overline{Z} Z^2)$,\\
    \item $(C(3) - W^* M(2) W)(\overline{Z}^3) =  \frac{\gamma_{33} -a}{\gamma_{24} -\overline{b}} (C(3) - W^* M(2) W)(\overline{Z}^2 Z),$\\
    \item the columns $(C(3) - W^* M(2) W)(Z^3)$ and $(C(3) - W^* M(2) W)(\overline{Z}^3)$ are nonlinear (because $(i)$ can not be verified).
  \end{itemize}
  Therefore, in $M(3)$, the columns $\overline{Z} Z^2$ is a linear combination of the columns $\mathfrak{B}(2) \cup \{ Z^3 \}$. For reason of simplicity, we adopt the notation of the Relation \eqref{line12}, that is,
 \begin{equation} \label{line12;1}
 \overline{Z} Z^2 = P_{\overline{z} z^2} (Z, \overline{Z}) = \alpha Z^3 + R_{\overline{z} z^2} (Z, \overline{Z}),
 \end{equation}
 Where
$$\alpha =\frac{\gamma_{42} -b}{\gamma_{33} -a} \neq 0$$
by using \eqref{alpha3}.

  We conclude that, in the both cases {\it II-1} and  {\it II-2}, we have extended the initial data $\gamma^{(5)}$ to $\gamma^{(6)}$ so that the associated moment matrix $M(3)$ has the following columns relation
\begin{equation} \label{line12;3}
 \overline{Z} Z^2 = P_{\overline{z} z^2} (Z, \overline{Z}) = \alpha Z^3 + R_{\overline{z} z^2} (Z, \overline{Z}), ~~~~\text{ with } \alpha
 \neq 0.
 \end{equation}
 We also note that since $a\neq e$ we get,
\begin{equation}
|\alpha|=\left|\frac{\gamma_{42} -b}{\gamma_{33}
-a}\right|=\frac{\sqrt{(\gamma_{33} -a)(\gamma_{33}
-e)}}{\gamma_{33} -a}\neq 1
\end{equation}

  {\it Step 2: ($M(3)$ has a flat extension, and thus a representing measure).}
  We will build  moments $\{ \gamma_{ij}\}_{i +j =7, 8}$ for which the moment matrix $M(4)$ is a flat extension of $M(3)$.

  The relation \eqref{line12;3} yields that
  $$\langle M(3) \overline{Z} Z^2, \overline{Z}^j Z^i\rangle = \langle M(3) P_{\overline{z} z^2} (Z, \overline{Z}), \overline{Z}^j Z^i\rangle,~~~~ \text{ for all } i +j\leq 3.$$
  By applying \eqref{5}, one obtain
  \begin{equation}\label{(6)}
  \Lambda_{\gamma^{(6)}}(\overline{z}^{i +1} z^{j +2}) = \Lambda_{\gamma^{(6)}}(\overline{z}^i z^j P_{\overline{z} z^2}), \phantom{ttttttt} i +j\leq 3.
  \end{equation}

     Since $|\alpha| \neq 1$, we derive that  there exists a complex number $\gamma_{43} = \overline{\gamma_{43}}$ such that
  \begin{equation}\label{e(43)}
  \gamma_{43} = \Lambda (\overline{z}^3 z P_{\overline{z} z^2}),
  \end{equation}
  that is, $$\gamma_{43} = \alpha \gamma_{34} + \sum\limits_{i +j \leq 2} \alpha_{i, j} \gamma_{i +3, j+1}.$$

It follows, from \eqref{e(43)} and \eqref{(6)}, that $\overline{z}
z^2 -P_{\overline{z} z^2}$ is  a generating polynomial of
$\gamma^{(6)} \cup \{\gamma_{34}, \gamma_{43} \}$.

  Since   $\bigl( \begin{smallmatrix} M(2)_{\mid \mathfrak{B}(2) } & Z^3_{\mid \mathfrak{B}(2) }\\ (Z^3_{\mid \mathfrak{B}(2) })^* & \gamma_{33} \end{smallmatrix} \bigr) > 0$, then there exists a (unique) vector, say
  $$P_{z^4} = \beta z^3 + R_{z^4} = \beta z^3 +\sum\limits_{\overline{z}^i z^j \in \mathfrak{B}(2)} \beta_{ij} \overline{z}^i z^j$$
  the associated polynomial, such that
  $$\bigl( \begin{smallmatrix} M(2)_{\mid \mathfrak{B}(2) } & Z^3_{\mid \mathfrak{B}(2) }\\ (Z^3_{\mid \mathfrak{B}(2) })^* & \gamma_{33} \end{smallmatrix} \bigr) P_{z^4} = ((\gamma_{04}, \gamma_{14}, \gamma_{05}, \gamma_{24}, \gamma_{15}, \gamma_{06})_{\mid \mathfrak{B}(2) }, \gamma_{34})^T.$$

   Therefore the sequence $\gamma^{(6)} \cup \{ \gamma_{34}, \gamma_{43}\}$ verifies that

   \begin{equation}\label{r04}
   \gamma_{i, j +4} = \Lambda(\overline{z}^i z^j P_{z^4}), \phantom{tex} \text{ for all } i +j \leq 2 \text{ and } (i, j) = (3, 0);
   \end{equation}

   \begin{equation}\label{r40}
   \gamma_{i +4, j} = \Lambda(\overline{z}^i z^j P_{\overline{z}^4}), \phantom{tex} \text{ for all } i +j \leq 2 \text{ and } (i, j) = (0, 3).
   \end{equation}
    Thus $z^4 -P_{z^4}$ is a generating polynomial of $\gamma^{(6)} \cup \{ \gamma_{34}, \gamma_{43}\}$.

    We will build a sequence $\gamma^{(8)} \equiv \{\gamma_{ij} \}_{i+j \leq 8}$, the extension of $\gamma^{(6)} \cup \{ \gamma_{34}, \gamma_{43}\}$, by using a generating polynomial $P_{z^4}$ and the initial data $\{\gamma_{ij} \}_{i, j \leq 3}$, that is,
    \begin{equation}\label{mr(04)}
    \gamma_{i, j +4} = \Lambda(\overline{z}^i z^j P_{z^4}) = \beta \gamma_{i, j+3} +\sum\limits_{\overline{z}^l z^k \in \mathfrak{B}(2)} \beta_{lk} \gamma_{i+l, j+k} \phantom{text} (i +j \leq 4)
    \end{equation}
   or, equivalently,
   \begin{equation}\label{mr(40)}
    \gamma_{i +4, j} = \Lambda(\overline{z}^i z^j P_{\overline{z}^4}) = \overline{\beta} \gamma_{i +3, j} +\sum\limits_{\overline{z}^l z^k \in \mathfrak{B}(2)} \overline{\beta_{lk}} \gamma_{i+k, j+l} \phantom{text} (i +j \leq 4).
    \end{equation}

  Hence, lemma \ref{lem.II-1} implies that $\overline{z} z^2 -P_{\overline{z} z^2}$ and $z^4 -P_{z^4}$ are two generating polynomials of $\gamma^{(8)}$.
  Therefore, in $M(4)$, the columns $Z^4, \overline{Z} Z^3, \overline{Z}^2 Z^2, \overline{Z}^3 Z, \overline{Z}^4$ are a linear combination of the columns
  $\{\overline{Z}^i Z^j \}_{i+j \leq 3}$ and thus $M(4)$ is a flat extension of $M(3)$. Indeed, it suffices to observe that
 $P_{z^4}, P_{\overline{z}^4}, P_{\overline{z} z^2}, P_{\overline{z}^2 z} \in V \equiv Vect( Z^3, \overline{Z}^3, Z^2, \overline{Z}Z, \overline{Z}^2, \overline{Z}, Z, 1)$ and thus $zP_{\overline{z} z^2}, \overline{z} P_{\overline{z}^2 z}, \overline{z}P_{\overline{z}z^2} \in V$; also one have, for all $i +j \leq 4$,
  \begin{equation*}\begin{aligned}
&\langle M(4) Z^4, \overline{Z}^i Z^j\rangle = \langle M(4) P_{z^4}, \overline{Z}^i Z^j\rangle;\\
&\langle M(4) \overline{Z}^4, \overline{Z}^i Z^j\rangle = \langle M(4) P_{\overline{z}^4}, \overline{Z}^i Z^j\rangle;\\
&\langle M(4)\overline{Z}Z^3, \overline{Z}^i Z^j\rangle = \langle M(4) zP_{\overline{z} z^2}, \overline{Z}^i Z^j\rangle;\\
&\langle M(4)\overline{Z}^2Z^2, \overline{Z}^i Z^j\rangle = \langle M(4)\overline{z}P_{\overline{z}z^2}, Z^i Z^j\rangle ;\\
\text{ and } &\langle M(4)\overline{Z}^3Z, \overline{Z}^i Z^j\rangle
= \langle M(4) \overline{z} P_{\overline{z}^2 z}, \overline{Z}^i
Z^j\rangle.
 \end{aligned} \end{equation*}
   This finishes the proof of the theorem.
\section{Example}
We consider the quintic sequence,
$$
\begin{array}{llllll}
\gamma_{00}=6& & & & & \\ \gamma_{01}=1+i& \gamma_{10}=1-i& & & &\\
\gamma_{20}=-2i & \gamma_{11}=6 & \gamma_{02}=2i & & &\\
\gamma_{30}=-2-2i & \gamma_{21}=2-2i & \gamma_{12}=2+2i & \gamma_{03}=-2+2i & &\\
\gamma_{40}=0 & \gamma_{31}= -4i& \gamma_{22}=8 & \gamma_{13}=4i & \gamma_{04}=0 & \\
\gamma_{50}=-4+4i & \gamma_{41}=-4-4i & \gamma_{32}=4-4i &
\gamma_{23}=4+4i & \gamma_{14}=-4+4i &\gamma_{05}=-4-4i.
\end{array}
$$
then our matrices are
$$
M(2)= \left(
\begin{array}{cccccc}
 6 & 1+i & 1-i & 2 i & 6 & -2 i \\
 1-i & 6 & -2 i & 2+2 i & 2-2 i & -2-2 i \\
 1+i & 2 i & 6 & -2+2 i & 2+2 i & 2-2 i \\
 -2 i & 2-2 i & -2-2 i & 8 & -4 i & 0 \\
 6 & 2+2 i & 2-2 i & 4 i & 8 & -4 i \\
 2 i & -2+2 i & 2+2 i & 0 & 4 i & 8 \\
\end{array}
\right) $$ and
$$
 B=\left(
\begin{array}{cccc}
 -2+2 i & 2+2 i & 2-2 i & -2-2 i \\
 4 i & 8 & -4 i & 0 \\
 0 & 4 i & 8 & -4 i \\
 4+4 i & 4-4 i & -4-4 i & -4+4 i \\
 -4+4 i & 4+4 i & 4-4 i & -4-4 i \\
 -4-4 i & -4+4 i & 4+4 i & 4-4 i \\
\end{array}
\right).
$$
The fact that $M(2)$ is positive definite implies,
$$
W=(M(2))^{-1} B = \left(
\begin{array}{cccc}
 0 & 0 & 0 & 0 \\
 0 & 1 & 0 & 1 \\
 1 & 0 & 1 & 0 \\
 \frac{3}{4}+\frac{3 i}{4} & \frac{1}{4}-\frac{i}{4} & -\frac{1}{4}-\frac{i}{4} & -\frac{3}{4}+\frac{3 i}{4} \\
 0 & 0 & 0 & 0 \\
 -\frac{3}{4}-\frac{3 i}{4} & -\frac{1}{4}+\frac{i}{4} & \frac{1}{4}+\frac{i}{4} & \frac{3}{4}-\frac{3 i}{4} \\
\end{array}
\right)
$$
and
$$
W^*M(2)W=\left(
\begin{array}{cccc}
 12 & -8 i & -4 & 8 i \\
 8 i & 12 & -8 i & -4 \\
 -4 & 8 i & 12 & -8 i \\
 -8 i & -4 & 8 i & 12 \\
\end{array}
\right)
$$
Since $a=e=12$ and $b=f=-8i$, according to theorem \ref{mainlast},
our sequence is a moment matrix for a 6 atomes measure. In fact,
from $W$, we can see that
$Z^3+\frac{3(1+i)}{4}(\bar{Z}^2-Z^2)-\bar{Z}$ and
$Z^2\bar{Z}+\frac{(1-i)}{4}(\bar{Z}^2-Z^2)-Z$ are two characteristic
polynomials for the moment sequence. The comment roots of the two
polynomials are
$$
\{\pm 1 , \pm i , 0 , 1+i \}
$$
Finally get that
$\mu=\delta_{1}+\delta_{-1}+\delta_{i}+\delta_{-1}+\delta_{0}+\delta_{1+i}$.



\end{document}